\documentclass[10pt]{article}

\usepackage[dvipdfmx]{graphicx,color}
\usepackage{latexsym}
\usepackage{amsmath}
\usepackage{amssymb}
\usepackage{amsthm}
\usepackage{ascmac}
\usepackage{bm}
\usepackage{graphicx}
\usepackage{wrapfig}
\usepackage{fancybox}

\newtheorem{theorem}{Theorem}[section]

\newtheorem{lemma}[theorem]{Lemma}
\newtheorem{proposition}[theorem]{Proposition}
\theoremstyle{definition}
\newtheorem{definition}[theorem]{Definition}
\theoremstyle{remark}
\newtheorem*{rem}{Remark}

\newcommand\C{\mathbb{C}}
\newcommand\K{\mathbb{K}}

\newcommand\R{\mathbb{R}}

  \makeatletter
    
    \@addtoreset{equation}{section}
  \makeatother

\begin{document}

\title{Fold singularities of the maps associated with Milnor fibrations for mixed polynomials}
\author{Daiki Sumida}
\date{}
\maketitle
\begin{abstract}
Milnor fibrations were extended by Mutsuo Oka for certain mixed polynomial. In this paper, we study singular points of differentiable maps into the 2-dimensional torus, called Milnor fibration product maps, obtained by several Milnor fibrations for mixed polynomial. We give a characterization of singular points of such product maps, and for the case of certain polar weighted homogeneous polynomials, a criterion for a fold singular point.
\end{abstract}
\section{Introduction}

\hspace{5mm}Let $S_\varepsilon$ be the sphere centered at the origin with radius $\varepsilon>0$, let $K_f$ be the intersection of the sphere $S_\varepsilon$ and the complex hypersurface $f^{-1}(0)$ where $f$ is an $n$ variable complex polynomial (resp. $K_g$, $K_{fg}$). Then
\begin{eqnarray*}
\phi_f = \frac{f}{|f|}:S_\varepsilon \setminus K_f \to S^1
\end{eqnarray*}
is a locally trivial fibration which is called {\it Milnor fibration} for sufficiently small radius $\varepsilon$. Under the theory of Milnor fibration, we studied the singularity of the $C^{\infty}$ maps which is called {\it Milnor fibration product map} (MFPM for short) defined by
\begin{eqnarray*}
\Phi = \left(\frac{f}{|f|},\frac{g}{|g|}\right):S_\varepsilon \setminus K_{fg} \to S^1\times S^1 = T^2
\end{eqnarray*}
in our doctoral report. Especially, we gave a necessary and sufficient condition for a point to be a singular point of the Milnor fibration product map and moreover gave a necessary and sufficient condition for a singular point of the Milnor fibration product map to be a fold point about certain weighted homogeneous polynomials $f$ and $g$.

In this paper, we could extend from the result about the singularity of the Milnor fibration product map for complex polynomials to it for mixed polynomials to use similarly technique in the original result. At first we give a necessary and sufficient condition for a point to be a singular point of the Milnor fibration product map associated with mixed polynomials $f$ and $g$ in general (Proposition \ref{2}). Next we refine this condition in certain polar weighted homogeneous polynomials $f$ and $g$ (Proposition \ref{3}). Mutsuo Oka gave a necessary condition for a mixed polynomial $f$, $\phi_f : S_\varepsilon \setminus K_f \to S^1$ become a fibration for sufficiently small radius $\varepsilon$ like Milnor fibration with complex polynomial $f$. But for generally mixed polynomial $f$, $\phi_f : S_\varepsilon \setminus K_f \to S^1$ may not become fibration. In the case of polar weighted homogeneous polynomials $f$ and $g$ such that $\phi_f : S_\varepsilon \setminus K_f \to S^1$ and $\phi_g : S_\varepsilon \setminus K_g \to S^1$ are fibrations (which include the case that Oka gave),  we give a necessary and sufficient condition for a singular point of the Milnor fibration product map associated with $f$ and $g$ to be a fold point (Theorem \ref{4}). This theorem is main result in this paper.

\section{Preliminaries}

\subsection{Notation}

At first, we let some definitions as Oka \cite{Oka3}.

\begin{definition}
We consider a complex polynomial $f(z,\overline{z})=\sum_{\nu,\mu}c_{\nu,\mu}z^{\nu}\overline{z}^{\mu}$ where $z=\left(z_1,\ldots,z_n\right), \overline{z} = \left(\overline{z_1},\ldots,\overline{z_n}\right), z^{\nu} = z_1^{\nu_1}\cdots z_n^{\nu_{n}}$ for $\nu=\left(\nu_1,\ldots,\nu_{n}\right)$, $\overline{z}^{\mu}=\overline{z_1}^{\mu_1}\cdots\overline{z_n}^{\mu_{n}}$ for $\mu=\left(\mu_1,\ldots,\mu_{n}\right)$. We call $f(z,\overline{z})$ a {\it mixed polynomial} of $n$ variables $z_1,\ldots,z_n$.
Let $w=(w_1,\ldots,w_n)$ be a integer vector satisfying ${\rm gcd}(w_1,\ldots,w_n)=1$, let $M$ be a non-zero monomial defined by $M=c_{\nu,\mu} z^\nu\overline{z}^\mu$ and let $d$ be a positive integer. Put ${\rm rdeg}_w M=\sum_{j=1}^n w_j(\nu_j+\mu_j)$ and put ${\rm pdeg}_w M=\sum_{j=1}^n w_j(\nu_j-\mu_j)$. A polynomial $f(z,\overline{z})$ is called a {\it radially weighted homogeneous polynomial} with {\it type} $(w;d)$ if ${\rm rdeg}_w M=d$ for every non-zero monomial $M$ of $f(z,\overline{z})$. The vector $w$ is called the {\it radial weight vector}, $d$ is called {\it radial degree}. Respectively a polynomial $f(z,\overline{z})$ is called a {\it polar weighted homogeneous polynomial} with {\it type} $(w;d)$ if ${\rm pdeg}_w M=d$ for every non-zero monomial $M$ of $f(z,\overline{z})$. The vector $w$ is called the {\it polar weight vector}, $d$ is called {\it polar degree}. Note that $f(z,\overline{z})$ be a polar weighted homogeneous polynomial with type $(w;d)$ if and only if we have

\begin{eqnarray*}
f\left(\lambda^{w_1}z_1,\ldots,\lambda^{w_{n}}z_n,\overline{\lambda}^{w_1}\overline{z_1},\ldots,\overline{\lambda}^{w_{n}}\overline{z_n}\right)=\lambda^{d}f\left(z,\overline{z}\right),\ \lambda \in \C,\ |\lambda|=1.
\end{eqnarray*}
\end{definition}

In this paper, we use the following notations.

\begin{definition}
For a holomorphic function $h$ of $z_1,\ldots,z_n,\overline{z_1},\ldots,\overline{z_n}$, we define
\begin{eqnarray*}
&&dh = \left(\frac{\partial h}{\partial z_1} , \ldots , \frac{\partial h}{\partial z_n}\right),\ \overline{d}h = \left(\frac{\partial h}{\partial\overline{z_1}}, \ldots , \frac{\partial h}{\partial\overline{z_n}}\right),\\
&&H^{\alpha,\beta}(h)=\left(\frac{\partial^2 h}{\partial \alpha_j \partial \beta_k}\right)_{j,k},\ H_p^{\alpha, \beta}(h) = \lim_{z\to p} H^{\alpha,\beta}(h),\\
&&v_h(z) = i(\overline{d\log h}-\overline{d}\log h)(z,\overline{z}).
\end{eqnarray*}
\end{definition}

\begin{definition}
Let $O$ be the origin, let ${}^t V$ be the transposed matrix of $V$, let $\overline{V}$ be the complex conjugate matrix of $V$, let $\textrm{Re}V$ be the real part matrix of $V$, let $\K^{m*}=\K^m \setminus \{O\}$ for $\K=\R$ or $\C$, let $\R^{+m}=[0,\infty)^m$. We denote $\left\langle\cdot,\cdot\right\rangle$ the Hermitian inner product.
\end{definition}

\subsection{Milnor fibration for mixed polynomial}

In this subsection, we introduce some results about Milnor fibration for mixed polynomial. We agree with the definition of Oka \cite{Oka3}.

\begin{definition}
Let $S_\varepsilon$ be the sphere defined by $S_\varepsilon = \{\left. z \in \C^n\,\right|\,\|z\|=\varepsilon\}$, let $f(z,\overline{z})$ be a mixed polynomial with $O \in f^{-1}(0)$, let $K_f$ be the intersection of the sphere $S_\varepsilon$ and the {\it mixed hypersurface} $f^{-1}(0)$,
 let $\phi_f$ be a map defined by $\phi_f = f/|f| : S_\varepsilon \setminus K_f \to S^1$.
\end{definition}

Oka gave a necessary and sufficient condition for a point to be a critical point of the map $\phi_f$.

\begin{lemma}
A point $p \in S_\varepsilon \setminus K_f$ is a singular point of $\phi_f$ if and only if $p, v_f(p)$ are linearly dependent over $\R$.
\end{lemma}

Moreover, Oka decide a necessary condition for a mixed polynomial $f$, $\phi_f : S_\varepsilon \setminus K_f \to S^1$ becomes a fibration to use Newton boundary.

\begin{definition}[Newton boundary of a mixed polynomial]
For a mixed polynomial $f(z,\overline{z})=\sum_{\nu,\mu}c_{\nu,\mu}z^\nu\overline{z}^\mu$, the {\it radial Newton polygon} ({\it Newton polygon} for short) $\Gamma_+(f;z,\overline{z})$ of a mixed polynomial $f(z,\overline{z})$ which is defined by the convex hull of
\begin{eqnarray*}
\bigcup_{c_{\nu,\mu}\neq 0}\left(\nu+\mu\right) + \R^{+n}.
\end{eqnarray*}
The {\it Newton boundary} $\Gamma \left(f;z,\overline{z}\right)$ is defined by the union of compact faces of $\Gamma_{+}\left(f;z,\overline{z}\right)$. For a given positive integer vector $w=(w_1,\ldots,w_n)$, we associate a linear function 
$\ell_w$ on $\Gamma(f)$ defined by 
$\ell_{w}\left(\nu\right)=\sum_{j=1}^{n}w_{j}\nu_{j}$ for $\nu \in \Gamma(f)$ and let $\Delta\left(w,f\right)=\Delta\left(w\right)$ be the face where $\ell_w$ takes its minimal value. We denote by $N$ the set of integer weight vectors and denote a vector $w \in N$ by a column vectors. We denote by $N^+,\ N^{++}$ the subset of positive or strictly positive weight vectors respectively. Thus $w={}^{t}(w_1,\ldots,w_n) \in N^{++}$ (respectively $w\in N^+$) if and only if $w_i>0$ (respectively $w_i \geq 0$) for any $i=1,\ldots,n$. Let
\begin{eqnarray*}
d(w)=d\left(w;f\right)=\min\left\{\left. {\rm rdeg}_{w} z^{\nu} \overline{z}^{\mu}\ \right|\ c_{\nu,\mu}\neq 0\right\}.
\end{eqnarray*}
For a positive weight $w$, we define {\it the face function} $f_w(z,\overline{z})$ by
\begin{eqnarray*}
f_{w}\left(z,\overline{z}\right)=\sum_{\nu+\mu\in\Delta\left(w\right)}c_{\nu,\mu}z^{\nu}\overline{z}^{\mu}.
\end{eqnarray*}
\end{definition}

\begin{definition}[Non-degenerate mixed polynomial]
For $w \in N^{++}$, the face function $f_w(z,\overline{z})$ is a radially weighted homogeneous polynomial of type $(w;d)$ with $d=d(w;f)$. Let $w$ be a strictly positive weight vector. We say that $f(z,\overline{z})$ is {\it non-degenerate} for $w$, if the fiber $f_w^{-1}(0)\, \cap\, \C^{*n}$ contains no critical point of the mapping $f_w : \C^{*n} \to \C$. We say that $f(z,\overline{z})$ is \textit{strongly non-degenerate} for $w$ if the mapping $f_w : \C^{*n} \to \C$ has no critical points. If $\dim\Delta\left(w\right)\geq 1$, we further assume that $f_w : \C^{*n} \to \C$ is surjective onto $\C$. A mixed polynomial $f(z,\overline{z})$ is called {\it non-degenerate} (respectively {\it strongly non-degenerate}) if $f(z,\overline{z})$ is non-degenerate (respectively {\it strongly non-degenerate}) for any strictly positive weight vector $w$. We say that a mixed polynomial $f(z,\overline{z})$ is a {\it true non-degenerate function} if it satisfies further the non-emptiness condition:
\begin{center}
$(NE)$ : $\forall w \in N^{++}$ with $\dim\Delta\left(w,f\right)\geq 1$, the fiber $f_w^{-1}(0)\, \cap\, \C^{*n} \neq \emptyset$.
\end{center}
\end{definition}
\begin{rem}
Assume that $f(z)$ is a holomorphic function. Then $f_w(z)$ is a weighted homogeneous polynomial and we have the Euler equality:
\begin{eqnarray*}
d\left(w;f\right)f_{w}\left(z\right)=\sum_{j=1}^{n}w_{j}z_j\frac{\partial f_{w}}{\partial z_j}\left(z\right). 
\end{eqnarray*}
Therefore $f_w : \C^{*n} \to \C$ has no critical point over $\C^*$. Thus $f(z,\overline{z})$ is non-degenerate for $w$ implies $f(z,\overline{z})$ is strongly non-degenerate for $w$. This is also the case if $f_w(z,\overline{z})$ is a polar weighted homogeneous polynomial.
\end{rem}

At last, we give a class of mixed polynomial.

\begin{definition}
We call a mixed polynomial $f(z,\overline{z})$ \textit{good} if there exists a positive number $\varepsilon_0$ such that for every $\varepsilon$ with $0<\varepsilon \leq \varepsilon$, $\phi_f = f/|f| : S_\varepsilon \setminus K_f \to S^1$ has no singular point. Note that the strongly non-degenerate mixed polynomials is an example of good mixed polynomial. For a good mixed polynomial and sufficiently small radius $\varepsilon$, $\phi_f = f/|f| : S_\varepsilon \setminus K_f \to S^1$ becomes a fibration which is called \textit{Milnor fibration} and every fiber is called \textit{Milnor fiber}. 
\end{definition}

\section{Results}

In this section, we give new results for the singularities of the Milnor fibration product map for mixed polynomial.
\begin{definition}
Let $f,g$ be mixed polynomials. We call $\Phi = (f/|f|,g/|g|) : S_\varepsilon \setminus K_{fg} \to T^2$ \textit{Milnor fibration product map associated with $f,g$}.
\end{definition}
\begin{proposition}\label{2}
A point $p \in S_\varepsilon \setminus K_{fg}$ is a singular point of the Milnor fibration product map $\Phi$ associated with $f,g$ if and only if $p,v_f(p),v_g(p)$ are linearly dependent over $\R$.
\end{proposition}

\begin{proof}
For a real analytic curve $c : (-\delta,\delta) \to S_\varepsilon \setminus K_{fg}$ with $c(0)=p$ and the tangent vector $v_c = dc/dt|_{t=0}$, we have
\begin{eqnarray*}
\left. \frac{d}{dt} \phi_f\left(c(t)\right)\right|_{t=0}
&=&\left. \frac{d}{dt}\exp\left(i{\rm Re}\left(-i\log f\left(c\left(t\right),\overline{c\left(t\right)}\right)\right)\right)\right|_{t=0}\\
&=&i\left. \frac{d}{dt}{\rm Re}\left(-i\log f\left(c\left(t\right),\overline{c\left(t\right)}\right)\right)\right|_{t=0}\frac{f(p,\overline{p})}{\left|f(p,\overline{p})\right|}\\
&=&i{\rm Re}\left\langle v_c,i\left(\overline{d\log f}-\overline{d}\log f\right)(p,\overline{p})\right\rangle\frac{f(p,\overline{p})}{\left|f(p,\overline{p})\right|}
\end{eqnarray*}
(respectively $d\phi_g(c(t))/dt|_{t=0}$). Moreover for $v \in T_p\left(S_\varepsilon \setminus K_{fg}\right)$, we obtain
\begin{eqnarray*}
d_p\Phi(v)=\left(i{\rm Re}\left\langle v,v_f(p)\right\rangle\frac{f(p,\overline{p})}{\left|f(p,\overline{p})\right|},i{\rm Re}\left\langle v,v_g(p)\right\rangle\frac{g(p,\overline{p})}{\left|g(p,\overline{p})\right|}\right).
\end{eqnarray*}
Therefore it follows that $p$ is a singular point of $\Phi$ if and only if there exists $(\beta,\gamma) \in \R^{2*}$ such that for any $v \in T_p \left(S_\varepsilon \setminus K_{fg}\right)$, it follows that
\begin{eqnarray*}
{\rm Re}\left\langle v,\beta v_f(p) + \gamma v_g(p) \right\rangle=\beta{\rm Re}\left\langle v, v_f(p) \right\rangle + \gamma {\rm Re}\left\langle v, \gamma v_g(p) \right\rangle=0.
\end{eqnarray*}
Hence we have $p,v_f(p),v_g(p)$ are linearly dependent over $\R$.

Conversely if $p,v_f(p),v_g(p)$ are linearly dependent over $\R$, then there exists $(\beta,\gamma) \in \R^{2*}$ such that for every $v \in T_p \left(S_\varepsilon \setminus K_{fg}\right)$, it follows that $\beta{\rm Re}\left\langle v, v_f(p) \right\rangle + \gamma {\rm Re}\left\langle v, v_g(p) \right\rangle=0$. With the above, we have the desired conclusion.
\end{proof}

\begin{proposition}\label{3}
Let $f,g$ be polar weighted homogeneous polynomials with type $(w_f;d_f),(w_g;d_g)$ respectively. Suppose that $w_f$ is a strictly positive vector or a strictly negative vector (respectively $w_g$) and $d_g w_f = s d_f w_g$. A point $p$ is a singular point of the Milnor fibration product map $\Phi$ associated with $f,g$ if and only if $v_f(p),v_g(p)$ are linearly dependent over $\C$.
\end{proposition}

\begin{proof}
Let $h_{f,t},h_{g,t} : S_\varepsilon \setminus K_{fg} \to S_\varepsilon \setminus K_{fg}$ be the diffeomorphisms defined by
\begin{eqnarray*}
h_{f,t}\left(z\right)&=&\left(z_1\exp\left(\frac{iw_{f,1}t}{d_{f}}\right),\ldots,z_n\exp\left(\frac{iw_{f,n}t}{d_{f}}\right)\right),\\
h_{g,t}\left(z\right)&=&\left(z_1\exp\left(\frac{iw_{g,1}t}{d_{g}}\right),\ldots,z_n\exp\left(\frac{iw_{g,n}t}{d_{g}}\right)\right)
\end{eqnarray*}
where $w_f=(w_{f,1},\ldots,w_{f,n}),\ w_g=(w_{g,1},\ldots,w_{g,n})$. Note that $h_{g,st}=h_{f,t}$. For a singular point $p$ of $\Phi$, let $(\alpha,\beta,\gamma) \in \R^{3*}$ be a vector satisfying
\begin{equation}
\alpha p + \beta v_f(p) + \gamma v_g(p) = 0. \label{eq:1}
\end{equation}
Let $v(p)$ be the tangent vector defined by $v(p)=dh_{f,t}(p)/dt|_{t=0}$. Then we have
\begin{eqnarray*}
&&-i \left\langle v(p),p\right\rangle=\frac{1}{d_{f}}\sum_{j=1}^{n}w_{f,j}\left|p_{j}\right|^{2}\gtrless 0,\\
&&\left\langle v(p),v_f(p)\right\rangle=1,
\ \left\langle v(p),v_g(p)\right\rangle= s
\end{eqnarray*}
since
\begin{eqnarray*}
\left\langle v(p),v_f(p)\right\rangle
&=&\left. \frac{d}{dt}\left(-i\log f\left(h_{f,t}(p),\overline{h_{f,t}(p)}\right)\right)\right|_{t=0}\\
&=&\left. \frac{d}{dt}\left(-i\log\left(e^{it}f(p,\overline{p})\right)\right)\right|_{t=0}=1,\\
\left\langle v(p),v_g(p)\right\rangle
&=&\left. \frac{d}{dt}\left(-i\log g\left(h_{f,t}(p),\overline{h_{f,t}(p)}\right)\right)\right|_{t=0}\\
&=&\left. \frac{d}{dt}\left(-i\log g\left(h_{g,st}(p),\overline{h_{g,st}(p)}\right)\right)\right|_{t=0}\\
&=&\left. \frac{d}{dt}\left(-i\log\left(e^{ist}g(p,\overline{p})\right)\right)\right|_{t=0}= s.
\end{eqnarray*}
From the equation (\ref{eq:1}), we have the following equation 
\begin{eqnarray*}
&& \frac{i \alpha}{d_{f}}\sum_{j=1}^{n}w_{f,j}\left|p_{j}\right|^{2}
+ \beta + \gamma s\\
&=& \alpha\left\langle v(p),p\right\rangle+\beta\left\langle v(p),v_f(p)\right\rangle+\gamma\left\langle v(p),v_g(p)\right\rangle\\
&=&\left\langle v(p),\alpha p+\beta v_f(p)+\gamma v_g(p)\right\rangle=0,
\end{eqnarray*}
moreover these follow that $\alpha=0,\ \gamma = -\beta s$ by comparing the real parts and the imaginary parts of both sides of the equation. Therefore we have that $v_f(p),v_g(p)$ are linearly dependent over $\C$.

Conversely, let $(\beta',\gamma') \in \C^{2*}$ be a vector satisfying $\beta' v_f(p) + \gamma' v_g(p) = 0$. From this equation, we have
\begin{eqnarray*}
\beta'+ \gamma' s
&=&\beta'\left\langle v_f(p),v(p)\right\rangle+\gamma'\left\langle v_g(p),v(p)\right\rangle\\
&=&\left\langle\beta'v_f(p)+\gamma'v_g(p),v(p)\right\rangle=0,
\end{eqnarray*}
moreover we have $\beta' = -\gamma' s \neq 0$. Therefore $\beta v_f(p) + \gamma v_g(p) = 0$ holds for $(\beta,\gamma) = \overline{\gamma'}(\beta',\gamma') = (-|\gamma'|^2 s , |\gamma'|^2) \in \R^{2*}$. This equation says that $v_f(p),v_g(p)$ are linearly dependent over $\R$, in other words $p$ is a singular point of $\Phi$. With the above, we have the desired conclusion.
\end{proof}

\begin{theorem}\label{4}
Assuming as in Proposition \ref{3} and suppose that polynomials $f,g$ are good. A singular point $p$ of $\Phi$ is a fold point of $\Phi$ if and only if the matrix $M_p(V)$:
\begin{eqnarray*}
&&{\rm Re}\left({}^t VH_p^{z,z}\left(-i\left(\log g-s\log f\right)\right)V\right)\\
&+&{\rm Re}\left({}^t VH_p^{z,\overline{z}}\left(-i\left(\log g-s\log f\right)\right)\overline{V}\right)\\
&+&{\rm Re}\left({}^t \overline{V}H_p^{\overline{z},z}\left(-i\left(\log g-s\log f\right)\right)V\right)\\
&+&{\rm Re}\left({}^t \overline{V}H_p^{\overline{z},\overline{z}}\left(-i\left(\log g-s\log f\right)\right)\overline{V}\right)
\end{eqnarray*}
is a regular matrix where $V = ({}^t v_1,\ldots,{}^t v_{2n-2})$ is an $n \times (2n-2)$ matrix such that $\{v_1,\ldots,v_{2n-2}\}$ is a real basis of $T_p F_f = \left\{v \in \C^n\,|\,{\rm Re}\left\langle v,p\right\rangle={\rm Re}\left\langle v,v_f(p)\right\rangle=0\right\}$.
\end{theorem}

\begin{proof}
Suppose that the radius $\varepsilon$ satisfies $\phi_f$ has no singular points. Then we have every Milnor fiber is a real $(2n-2)$ dimensional manifold. Let $(x_1,\ldots,x_{2n-1})$ be a local coordinates of $S_\varepsilon \setminus K_f$ around $p$ such that $(x_1,\ldots,x_{2n-2})$ be a local coordinate of $F_f$ around $p$. Then we have that $p \in S_\varepsilon \setminus K_{fg}$ with $p \in F_f$ is a fold point of $\Phi$ if and only if $p$ is a non-degenerate critical point of argument map ${\rm Re}\left(-i\left(\log g-s\log f\right)\right)$ on $F_f$ as \cite{Sumida}. Let $p$ be a singular point of $\Phi$ and let $(x_1,\ldots,x_{2n-2})$ be a local coordinates of $F_f$ around $p$. We have
\begin{eqnarray*}
&&\frac{\partial^{2}}{\partial x_{c}\partial x_{d}}{\rm Re}\left(-i\left(\log g-s\log f\right)\right)\\
&=&\frac{\partial}{\partial x_{d}}{\rm Re}\left\langle\frac{\partial z}{\partial x_{c}},v_g-sv_f\right\rangle\\
&=&{\rm Re}\left\langle\frac{\partial^{2}z}{\partial x_{c}\partial x_{d}},v_g-sv_f\right\rangle+{\rm Re}\left\langle\frac{\partial z}{\partial x_{c}},\frac{\partial}{\partial x_{d}}\left(v_g-sv_f\right)\right\rangle.
\end{eqnarray*}
Note that $v_g - s v_f$ vanishes at every singular point of $\Phi$. We get
\begin{eqnarray*}
v_g-sv_f&=&i\left(\overline{d\log g}-\overline{d}\log g\right)-si\left(\overline{d\log f}-\overline{d}\log f\right)\\
&=&i\overline{d\left(\log g-s\log f\right)}-i\overline{d}\left(\log g-s\log f\right),\\
\overline{v_g-sv_f}
&=&-i d\left(\log g-s\log f\right)+ i d\overline{\left(\log g-s\log f\right)}\\
&=&d (-i \left(\log g-s\log f\right)+ i\overline{\left(\log g-s\log f\right)})
\end{eqnarray*}
since $\overline{\overline{d}h}=d\overline{h}$ for a holomorphic function $h$.\\
Moreover it follows that
\begin{eqnarray*}
&&{\rm Re}\left\langle\frac{\partial z}{\partial x_{c}},\frac{\partial}{\partial x_{d}}\left(v_g-sv_f\right)\right\rangle\\
&=&{\rm Re} \sum_{a=1}^{n}\frac{\partial z_a}{\partial x_{c}}\sum_{b=1}^{n}\left(\frac{\partial z_b}{\partial x_{d}}\frac{\partial}{\partial z_b}+\frac{\partial\overline{z_b}}{\partial x_{d}}\frac{\partial}{\partial\overline{z_b}}\right)\frac{\partial}{\partial z_a}\left(-i\left(\log g-s\log f\right)\right)\\
&+&{\rm Re} \sum_{a=1}^{n}\frac{\partial z_a}{\partial x_{c}}\sum_{b=1}^{n}\left(\frac{\partial z_b}{\partial x_{d}}\frac{\partial}{\partial z_b}+\frac{\partial\overline{z_b}}{\partial x_{d}}\frac{\partial}{\partial\overline{z_b}}\right)\frac{\partial}{\partial z_a}\left(i\overline{\left(\log g-s\log f\right)}\right)\\
&=&{\rm Re} \sum_{a=1}^{n}\sum_{b=1}^{n}\left(\frac{\partial z_a}{\partial x_{c}}\frac{\partial z_b}{\partial x_{d}}\frac{\partial^{2}}{\partial z_a \partial z_b}+\frac{\partial z_a}{\partial x_{c}}\frac{\partial\overline{z_b}}{\partial x_{d}}\frac{\partial^{2}}{\partial z_a\partial\overline{z_b}}\right)\left(-i\left(\log g-s\log f\right)\right)\\
&+&{\rm Re} \sum_{a=1}^{n}\sum_{b=1}^{n}\left(\frac{\partial z_a}{\partial x_{c}}\frac{\partial z_b}{\partial x_{d}}\frac{\partial^{2}}{\partial z_a \partial z_b}+\frac{\partial z_a}{\partial x_{c}}\frac{\partial\overline{z_b}}{\partial x_{d}}\frac{\partial^{2}}{\partial z_a\partial\overline{z_b}}\right)\left(i\overline{\left(\log g-s\log f\right)}\right)\\
&=&{\rm Re} \sum_{a=1}^{n}\sum_{b=1}^{n}\left(\frac{\partial z_a}{\partial x_{c}}\frac{\partial z_b}{\partial x_{d}}\frac{\partial^{2}}{\partial z_a \partial z_b}+\frac{\partial z_a}{\partial x_{c}}\frac{\partial\overline{z_b}}{\partial x_{d}}\frac{\partial^{2}}{\partial z_a\partial\overline{z_b}}\right)\left(-i\left(\log g-s\log f\right)\right)\\
&+&{\rm Re} \sum_{a=1}^{n}\sum_{b=1}^{n}\left(\frac{\partial\overline{z_a}}{\partial x_{c}}\frac{\partial\overline{z_b}}{\partial x_{d}}\frac{\partial^{2}}{\partial\overline{z_a}\partial\overline{z_b}}+\frac{\partial\overline{z_a}}{\partial x_{c}}\frac{\partial z_b}{\partial x_{d}}\frac{\partial^{2}}{\partial\overline{z_a}\partial z_b}\right)\left(-i\left(\log g-s\log f\right)\right).
\end{eqnarray*}
Therefore we have
\begin{eqnarray*}
\left(\frac{\partial^{2}}{\partial x_{c}\partial x_{d}}{\rm Re}\left(-i\left(\log g-s\log f\right)\right)(0)\right)_{c,d=1}^{2n-2}
=M_p(V_0)
\end{eqnarray*}
where $V_{0}=\left(\partial z_j/\partial x_{k}(0)\right)_{j,k}$ is a full rank $n \times (2n-2)$ matrix. For every local coordinates $(x_1,\ldots,x_{2n-2})$, the rank of this Hessian matrix is constant. Therefore, for every real basis $\{v_1,\ldots,v_{2n-2}\}$ of $T_p F_f$ and $V=(v_1,\ldots,v_{2n-2})$, the rank of $M_p(V)$ is constant. We have desired conclusion.

\end{proof}

\end{document}